\documentclass[12pt, a4paper]{amsart}
\usepackage{amscd}
\usepackage{amsmath}
\usepackage{latexsym}
\usepackage{graphicx}
\usepackage{amsfonts}
\usepackage{amssymb}
\nonstopmode
\newcommand{\Z}{\mathbb{Z}}

\numberwithin{equation}{section}

\theoremstyle{plain}
\newtheorem{theorem}[equation]{Theorem}

\newtheorem{corollary}[equation]{Corollary}
\newtheorem{proposition}[equation]{Proposition}
\newtheorem{lemma}[equation]{Lemma}
\theoremstyle{definition}

\newtheorem{example}[equation]{Example}

\newtheorem{defn}[equation]{Definition}

\theoremstyle{remark}
\newtheorem{remark}[equation]{Remark}

\newtheorem{rem}[equation]{Remark}
\newtheorem*{remm}{Remark}

\def\map#1{{\buildrel #1 \over \longrightarrow}}

\newcommand{\chr}{\operatorname{chr}}
\newcommand{\cone}{\operatorname{cone}}
\newcommand{\cyl}{\operatorname{cyl}}

\newcommand{\RP}{\mathbb{RP}}
\begin{document}
\title[On the covering type of a space]
{On the covering type of a space}
\date{\today}

\author{Max Karoubi}
\address{Institut Math\'ematiques de Jussieu--Paris rive gauche,
75205 Paris Cedex 13, France}
%Universit\'e Denis Diderot Paris 7, Institut Math\'ematique de Jussieu}
\email{max.karoubi@gmail.com}
\author{Charles Weibel}
\address{Math.\ Dept., Rutgers University, New Brunswick, NJ 08901, USA}
\email{weibel@math.rutgers.edu}
\thanks{Weibel was supported by NSA and NSF grants, 
and by the IAS Fund for Math} 

\begin{abstract}
%We introduce a notion of ``covering type'' of a space which measures
%its complexity. 
We introduce the notion of the ``covering type'' of a space, 
which is more subtle that the notion of Lusternik Schnirelman category. 
It measures the complexity of a space which arises from coverings by
contractible subspaces whose non-empty intersections are also contractible.
\end{abstract}
\maketitle

\pagestyle{myheadings} \setcounter{section}{1}%
From the point of view of Algebraic Topology, the simplest spaces are 
the contractible ones. A crude measure of the complexity of a space $X$
is the size of a finite open covering of $X$ by contractible subspaces;
this idea goes back to the work of Lusternick and Schnirelman in 1934.
%(closed or open; see Section \ref{sec:open/closed} for more details) 
A more subtle invariant is the size of {\it good covers}, i.e.,
covers by contractible subspaces such that each of their non-empty 
intersections is also contractible.
%\edit{Leray added}
The idea of a good cover appears in a 1952 paper by Andr\'e Weil \cite{Weil},
but is preceded by Leray's notion of a {\it convexoid cover} 
\cite[p.\,139]{Leray} which uses closed covers with acyclic intersections.

We define the {\it covering type} of $X$ to be the minimum size of 
a good cover of a space homotopy equivalent to $X$.
See Definition \ref{def:ct} below for a precise description
using open covers; for finite CW complexes, an equivalent version 
using closed covers by subcomplexes
is given in Theorem \ref{open=closed}.

We will see that this is an interesting measure of the complexity
of a space. For connected graphs, 
the covering type is approximately $\sqrt{2h}$,
where $h$ is the number of circuits. Thus it is very different from
the Lusternick--Schnirelman category (which is 2 for non-tree graphs),
and Farber's topological complexity \cite{Farber} (which is at most 3
for graphs).

The covering type of a surface is related to its chromatic number.
% at most one more than the chromatic
% number of the surface, and is at least half the chromatic number.
The {\it chromatic number} of a surface is the smallest number $n$ such that
every map on the surface is $n$-colorable (see Definition \ref{def:chromatic}); 
it was first described in 1890 by Haewood \cite{Hae}.
Finding the chromatic number of a surface was long known as the 
{\it map coloring problem}. The chromatic number of the 2-sphere is 4;
this is the ``four color theorem,'' settled in 1976.  The other cases
of the map coloring problem
were settled in 1968; see \cite{Ringel}. 
For the 2-sphere, torus and projective plane, the covering type
equals the chromatic number: 4, 7 and 6, respectively.

Special solutions to the map coloring problem give an upper bound
for the covering type of a surface: it is at most one more than 
the chromatic number. 
%The covering type equals the chromatic number
%for the 2-sphere, torus and projective plane: 4, 7 and 6, respectively
Combinatorial and topological considerations of any space, 
such as its Betti numbers, yield general lower bounds for its covering type.
% of any space come from combinatorial 
%and topological considerations, such as the Betti numbers of $X$.
This approach shows that the covering type of 
a surface is alway more than half the chromatic number.
Although our upper and lower bounds for the covering type of a surface 
differ as functions of its genus $g$, both are linear in $\sqrt{g}$
(see Section \ref{sec:surfaces}).

\subsection*{\bf Open problem:}
Except for the sphere, torus and projective plane,
determining the covering type of a surface is an open problem.  
For example, we do not know if the
covering type of the Klein bottle is 7 or 8. For the 2-holed torus
(oriented surface of genus~2), 
we only know that the covering type is between 6 and 10.

\goodbreak
%\smallskip
\subsection*{\bf Motivation:}
Leray's motivation for introducing his convexoid covers was to
    easily compute homology; see \cite[pp.153--159]{Leray}. 

Weil's motivation for introducing special open coverings
%({\it recouvrements diff\'erentiablement simple}) 
was to prove de Rham's theorem for a manifold \cite{Weil}.
Here is a formalization of his idea in the language of cohomology theories.
Let $h^*$ and $k^*$ be two cohomology theories and let 
$T: h^* \to k^*$ be a natural transformation such that the kernel 
and cokernel of the homomorphism $T_X:h^*(X) \to k^*(X)$ are of order 
at most $m$ when $X$ is a point. 
Using Mayer-Vietoris sequences and induction on $n$, 
we see that for a space $X$ of 
covering type $n$, the kernel and cokernel of $T_X$ have
order at most $m^{2^n}$. This general principle was
applied by Weil in the case where $h^*$ is singular cohomology,
$k^*$ is de Rham cohomology and $m = 0$ (Poincar\'e's lemma). 
We used the same idea with $m=2^t$ in a preliminary version 
of our paper \cite{KW}, %where $m$ is a power of 2, 
comparing the algebraic Witt group of a real algebraic variety 
with a purely topological invariant; this was in fact our
initial motivation for investigating the notion of covering type.

\smallskip
Here are our definitions of a good cover and the covering type:

\begin{defn}\label{def:goodcover}
Let $X$ be a topological space. %CW-complex. 
A family of contractible open subspaces %subcomplexes
$U_i$ forms a {\it good (open) cover} if every nonempty intersection 
$U_{i_1}\cap...\cap U_{i_n}$ is contractible.
\end{defn}
\begin{defn}\label{def:ct}
The \textit{strict covering type} of $X$ is the
minimum number of elements in a good cover of $X$.
The \textit{covering type} of $X$, $ct(X)$, is the
minimum of the strict covering types of spaces $X^{\prime }$ homotopy
equivalent to $X$. 
%If there is no such $n,$ we say that the covering type is infinite.
\end{defn}

Thus $ct(X)=1$ exactly when $X$ is contractible, and $ct(X)=2$ 
exactly when $X$ is the disjoint union of two contractible spaces.
It is easy to see that a circle has covering type 3, and only slightly 
harder to see that the 2-sphere and figure~8 have covering type 4 
(small neighborhoods of the faces of a tetrahedron give a good cover
of the 2-sphere).
%(a good cover is obtained by quartering the sphere).

The difference between the strict covering type and the covering
type is illustrated by bouquets of circles.

\begin{example}\label{ct(figure8)}
The strict covering type of a bouquet of $h$ circles is 
$h+2$, %$2h+1$,
since 3 subcomplexes are needed for each circle. 
For $h=2$ and $3$, $S^{1}\vee S^{1}$ and 
$S^{1}\vee S^{1}\vee S^{1}$ both have covering type at most 4, 
since Figure \ref{fig:2-3 circles} indicates good covers of 
homotopy equivalent spaces.
(The $U_i$ are small neighborhoods of the three outer edges,
$X_{1},X_{2},X_{3}$ and the inside curve $X_4$).
%are small neighborhoods of the three outer edges, 
%and $X_{4}$ is a neighborhood of the inside curve).
We will see in Proposition \ref{ct:bouquets} that the covering type 
is exactly~4 in both cases. 
\end{example}
\begin{center}
\begin{figure}
%$ct(X)=4$ for bouquets of 2 or 3 circles \newline
\setlength{\unitlength}{1cm} 
\begin{picture}(1,1.2)
\put(0,0){\line(0,1){1}} \put(-.7,.3){$X_2$}
\put(0,0){\line(1,0){1}} \put(.3,-.4){$X_1$}
\put(1,0){\line(-1,1){1.0}} \put(.5,.6){$X_3$}
\put(0,0){\line(1,1){0.5}}
%\put(0.5,0){\line(0,1){.5}}
%\put(0.5,0){\line(0,1){.5}}
\end{picture}
\qquad \hspace{2pc}
\begin{picture}(1,1)
\put(0,0){\line(0,1){1}} \put(-.7,.3){$X_2$}
\put(0,0){\line(1,0){1}} \put(.3,-.4){$X_1$}
\put(1,0){\line(-1,1){1.0}} \put(.5,.6){$X_3$}
\put(0,0.5){\line(1,0){.5}}
\put(0.5,0){\line(0,1){.5}}
\end{picture}\par
%\qquad \hspace{2pc}
%\begin{picture}(1,1)
%\put(0,0){\line(0,1){1}} \put(-.7,.3){$X_2$}
%\put(0,0){\line(1,0){1}} \put(.3,-.4){$X_1$}
%\put(1,0){\line(-1,1){1.0}} \put(.5,.6){$X_3$}
%\put(0,.5){\line(1,0){.5}}
%\put(.25,0){\line(0,1){.5}}
%\end{picture}\par
\caption{$ct(X)=4$ for bouquets of 2 or 3 circles}
\label{fig:2-3 circles}
\end{figure}
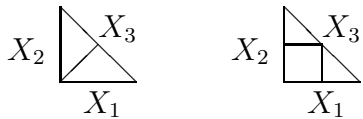
\end{center}
%

%In the next section, we show that the covering type of a finite
%CW complex may also be defined using covers by closed subcomplexes,
%instead of by open subsets. 
%This simplifies many arguments, and
%allows us to simplify our illustrations.

If $X$ is a CW complex, we may replace 'open subspace' 
by 'closed subcomplex' in Definition \ref{def:goodcover} to obtain
the analogous notion of a {\it good closed cover} of $X$,
and the concomitant notion of a closed covering type.  
We will show in Theorem \ref{open=closed} that the closed
and open covering types of a finite CW complex agree.
This simplifies 
%many arguments, and allows us to simplify 
our illustrations.

\smallskip
We have structured this article as follows.
In Section \ref{sec:open/closed}, we show that we may also define
the covering type of a finite CW complex using covers by contractible
subcomplexes.
%This simplifies the subsequent discussion.
In Section \ref{sec:small}, we establish useful lower bounds on
covering type using homology, and use these bounds in 
Section \ref{sec:graphs} to determine the covering type of any graph 
in Proposition \ref{ct:bouquets}.

In Sections \ref{sec:surfaces} and \ref{sec:cohomology} 
we determine the covering type of some classical surfaces, and
use graphs on an arbitrary surface to give upper bounds for its covering type.
For oriented surfaces of genus $g>2$, 
the covering type lies between $2\sqrt{g}$ and $3.5\sqrt{g}$;
see Proposition \ref{ct-surface}. Similar bounds hold for
non-oriented surfaces; see Proposition \ref{ct:non-oriented}.

We conclude in Section \ref{sec:higher}
with some upper bounds for the covering type in higher dimensions,
showing for instance that $ct(\RP^m)$ is between $m+2$ and $2m+3$.

\goodbreak\smallskip
\section{Open vs. closed covers}\label{sec:open/closed}

Any cover of a manifold by geodesically convex open subsets
is a good cover; by convexity, every nonempty intersection of them
is contractible. Thus compact manifolds have finite strict covering type.
(An elementary proof is given in \cite[VI.3]{Karoubi.Leruste}.)
%(There is a closely related notion using open
%coverings; see \cite{Bott/Tu}.)\\ {\bf update!}

Finite simplicial complexes also have finite strict covering type,
because the open stars of the vertices form a good cover.
Finite polyhedra also have finite strict covering type, as they
can be triangulated.
Any finite CW complex is homotopy equivalent to a finite
polyhedron by \cite[Thm.\,13]{Whitehead}, so
its covering type is also finite.

\begin{proposition}\label{CWcomplex}
If a paracompact space has finite covering type $n$, %if and only if 
it is homotopy equivalent to a finite CW complex whose strict covering
type is $n$.
\end{proposition}

\begin{proof}
Suppose that $X$ is a paracompact space with a finite
open cover $\mathcal{U}=\{U_i\}$ and let $N$ %=|\mathcal{U}|$
be its geometric nerve; $N$ is a simplicial complex whose vertices 
are the indices $i$, and a set $J$ of vertices spans a simplex of $N$ if\; 
$\bigcap_{i\in J} U_i\ne\varnothing$. 
The {\it Alexandroff map} $X\to N$ is a continuous function,
given by the following standard construction; see \cite[IX.3.4]{AH} or
\cite[VIII.5.4]{Dugundji}. 
Choose a partition of unity $\{v_i\}$ associated to the open cover.
Any point $x\in X$ determines a set $J_x$ of indices 
(the $i$ such that $x\in U_i$), and 
the Alexandroff map sends a point $x$ to
the point with barycentric coordinates $v_i(x)$ in the simplex 
spanned by $J_x$.

If $\{U_i\}$ is a good cover, %so is $\mathcal{U}$ and the
the Alexandroff map $X\to N$ is a homotopy equivalence, by the
``Nerve Lemma'' (see \cite[4G.3]{Hatcher}).
If~ $\mathcal{U}$ has $n$ elements then $N$ has $n$ vertices, and the
open stars of these vertices form a good cover of $N$.  If $ct(X)=n$
then $N$ has no smaller good cover, and hence $N$ has
strict covering type $n$.
\end{proof}

\begin{example}\label{rem:ct=3}
If $ct(X)=3$, $X$ is either the disjoint union of 3 contractible sets
or is homotopy equivalent to the circle, by the Alexandroff map.
To see this, suppose that $\{U_{1},U_{2},U_{3}\}$ is a good cover of 
a connected $X$.
By a case by case inspection, we see that $H^1(X)=0$ and $ct(X)<3$ 
unless $U_{i}\cap U_{j}\ne\varnothing$ for all $i,j$ and 
$U_{1}\cap U_{2}\cap U_{3}=\varnothing.$ In this case,
the Alexandroff map $X\to S^1$ is a homotopy equivalence
and $\dim H^1(X)=1$.

Similarly, a case by case inspection shows that (up to homotopy)
the only connected spaces with covering type~4 are the 2-sphere 
and the bouquets of
2 or 3 circles illustrated in Figure \ref{fig:2-3 circles}.
\end{example}

\begin{remark}
The ``Hawaiian earring'' \cite[1.25]{Hatcher} is the union of the
circles $(x-1/n)^2+y^2=1/n^2$ in the plane; see Figure \ref{fig:earring}. 
%\edit{figure added}
It is a compact space which has no good open cover; its 
strict covering type is undefined. It follows from Proposition 
\ref{CWcomplex} and compactness that its covering type is also undefined.
\end{remark}

%%%%%%%%%%%%%%%%%% added 6/2/16 %%%%%%%%%%%%%%%%%%%
\begin{center}
\begin{figure}
%Hawaiian earring
\setlength{\unitlength}{1cm} 
\begin{picture}(1,1.2)
\put(.95,0.5){\circle{2.0}} %unavailable size
\put(.75,0.5){\circle{1.0}}
\put(.62,0.5){\circle{0.7}}
\put(.55,0.5){\circle{0.55}}
\put(.45,0.5){\circle{0.4}}
\put(.40,0.5){\circle{0.25}}
\put(.35,0.5){\circle{0.15}}
\end{picture}\par
\caption{Hawaiian earring}
\label{fig:earring}
\end{figure}
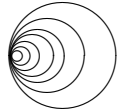
\end{center}
%%%%%%%%%%%%%%%%%%

\begin{defn}
If $X$ is a finite CW complex, a {\it good closed cover}
is a family of contractible subcomplexes $\{X_i\}$ such that
every intersection of the $X_i$ is either empty or contractible.
\end{defn}

For example, the maximal simplices of any simplicial complex
form a good closed cover. We will see more examples in Sections
\ref{sec:graphs} and \ref{sec:surfaces}.

For our next result, we need the classifying space $BP$ of a 
finite poset $P$. It is a simplicial complex whose vertices are the
elements of $P$, and whose simplices are the totally ordered
subsets of $P$.

\begin{theorem}\label{open=closed}
If $X$ is a finite CW complex, the covering type of $X$ is the
minimum number of elements in a good closed cover of 
some complex homotopy equivalent to $X$.
\end{theorem}

\begin{proof}
Suppose that $ct(X)=n$. By Proposition \ref{CWcomplex}, we may
% replace $X$ by a homotopy equivalent CW complex to 
suppose that $X$ is a simplicial complex with $n$ vertices.
We need to show that $X$ has a good closed cover with $n$ elements.
Let $X_i$ denote the closed subcomplex of $X$ consisting of 
points whose $i^{th}$ barycentric coordinate is $\ge1/n$. 
Then $\{X_i\}$ is a closed cover of $X$, because at least one
barycentric coordinate must be $\ge 1/n$. 
To see that it is a good closed cover, choose a subset $J$ 
of $\{1,...,n\}$ for which $X_J=\bigcap_{j\in J} X_j\ne\emptyset$. 
Then $X_J$ is contractible because, if $p$ is the point
whose $i^{th}$ barycentric coordinate is $1/|J|$ for $i\in J$ and
$0$ otherwise, the formula $h_t(x)=tx+(1-t)p$ defines
a deformation retraction from $X_J$ to the point $p$.
(To see that $h_t(x)$ is in $X_J$ for all $t$ with $0\le t\le1$,
note that for each $j\in J$, the $j^{th}$ coordinate of 
$h_t(x)$ is at least $1/n$.)
Thus $\{X_i\}_{i=1}^n$ is a good closed cover of $X$, as required.
%the strict covering type of $X$ is at most $n$, 
%and hence that $ct(X)\le n$.

Conversely, given a good closed cover $\{X_i\}_{i=1}^n$,
the nerve $N$ of this cover has $n$ vertices, 
and has a good open cover of size $n$.
We will construct a poset $P$ and
homotopy equivalences $X \leftarrow BP \to N$,
showing that $ct(X)\le n$. For this, we may assume $X$ is connected.

Let $S$ be the family of subsets $J$ of $\{1,...,n\}$ with
$X_J=\bigcap_{i\in S} X_i\ne\varnothing$; it is a poset 
($I\le J$ if $J\subseteq I$) and its realization $BS$ is the 
geometric nerve of the cover, $N$.
There is a functor from $S$ to contractible spaces,
sending $J$ to $X_J$: if $I\subseteq J$ then $X_J\subseteq X_I$.
Following Segal \cite[4.1]{S}, the disjoint union of all
%\edit{\cite[1.4]{S} now \cite[4.1]{S}}
nonempty intersections $X_J$ is a topological poset $P$, and
the obvious functor $P\to S$ yields a continuous function
from the geometric realization $BP$ to $N=BS$.
By the Lemma on p.\,98 of \cite{Q}
(the key ingredient in the proof of Quillen's Theorems~A and~B), 
$BP\to BS$ is a homotopy equivalence.
There is a canonical proper function $BP\map{p} X$ obtained from
the functor from $P$ to the trivial topological category with
$X$ its space of objects. It is an isomorphism on homology,
because the inverse image of a point $x\in X$ is a simplex, 
whose vertices correspond to the $i$ with $x\in X_i$.

Consider the universal covering space $\widetilde{X}\to X$; 
the inverse image of each $X_i$ is a disjoint union of spaces 
$X_{i,\alpha}$, each homeomorphic to $X_i$, and the 
$\{X_{i,\alpha}\}$ form a good cover of $\widetilde{X}$. 
(For each $i$, there is a non-canonical bijection between 
$\{(i,\alpha)\}$ and $\pi_1\,X$.) 
If $\widetilde{P}$ denotes the topological poset of intersections 
of the $\{X_{i,\alpha}\}$ then, as above,
$\tilde{p}: B\widetilde{P}\to\widetilde{X}$ is an isomorphism
on homology. Since both spaces are simply connected, the 
Whitehead Theorem implies that that $\tilde{p}$ is a homotopy equivalence.
By inspection, $B\widetilde{P}$ is the universal covering space 
of $BP$, so $\pi_1(BP)\cong\pi_1(\widetilde{X})$. 
This implies that $p$ is a homotopy equivalence, since for $n>1$ we have 
$\pi_nBP \cong \pi_nB\widetilde{P} \cong \pi_n\widetilde{X}\cong \pi_nX$.
\end{proof}

%For every object $x\in X$, the comma category $x/\tau$ is
%equivalent to the poset of all $J$ in $S$ with $x\in X_J$.
%As this poset has a minimal element, $(x/\tau)$ is contractible.
%By the topological version of Quillen's Theorem A \cite{Moerdijk},
%$BP\to X$ is a homotopy equivalence.
%{\bf Quillen's A needed to finish proof!}
%Choose a small open neighborhood 
%$U_i$ of each $X_i$ which deformation retracts onto $X_i$. 
%Then $\{U_i\}$ is a good open cover of $X$; this shows that
%$ct(X)\le n$.  

%\goodbreak
\section{Lower bounds for the covering type}\label{sec:small}

In general, the covering type of $X$ is not so easy to compute. 
A simple lower bound is provided by the proposition below, 
derived via a Mayer-Vietoris argument for the homology of $X$. 
We omit the proof, since it will follow from the more
general Theorem \ref{Poincare} below.

\begin{proposition}\label{bound:m+2}
Fix a field $k$ and
let $hd(X)$ denote the homological dimension of $X$, i.e.,
the maximum number such that $H_m(X,k)$ is nonzero. %for any field $k$.
Then, unless $X$ is  empty or contractible, 
\begin{equation*}
ct(X)\geq hd(X)+2.
\end{equation*}
\end{proposition}

%\smallskip

\begin{example}\label{sphere}
The sphere $S^{m}$ has $ct(S^{m})=m+2$. This is
clear for $S^{0}$ and $S^{1}$, and the general case follows from the
upper bound in Proposition \ref{bound:m+2}, combined with the
observation that $S^m$ is homeomorphic to the boundary of the
$(m\!+\!1)$-simplex, which has $m+2$ maximal faces.

Alternatively, we could get the upper bound from the
suspension formula $ct(SX)\leq 1+ct(X)$. This formula is a simple exercise
which will be generalized later on (Theorem \ref{thm:cone}).
\end{example}

%\bigskip

A stronger lower bound for $ct(X)$ uses the Poincar\'{e} polynomial%
\begin{equation*}
P_{X}(t)=h_{0}+h_{1}t+...+h_{m}t^{m}
\end{equation*}%
where $h_{i}$ is the rank of the homology $H_{i}(X)$ or 
cohomology $H^{i}(X)$ (with coefficients in a field). 
We partially order the set of polynomials in $%
\Z[t]$ by declaring that $P\leq Q$ if and only if all the
coefficients of $P$ are smaller or equal to the respective coefficients of $%
Q.$ We now have the following theorem:

\begin{theorem}\label{Poincare} 
Let $P_X(t)$ be the Poincar\'e polynomial of $X$ and let $n$
be its covering type. If $X$ is not empty then: 
\begin{equation*}
P_X(t)\leq\frac{(1+t)^{n-1}-1}{t}+1 = 
n + \binom{n-1}2 t + \binom{n-1}3 t^2 + \cdots+ t^{n-2}.
\end{equation*}
That is, $h_0\le n$, $h_1\le\binom{n-1}2$,...,$h_{n-2}\le1$ and $h_i=0$
%$n\ge h_0$, $\binom{n-1}2\ge h_1$, ..., $1\ge h_{n-2}$ and $h_i=0$
for $i\ge n-1$.
\end{theorem}

\begin{proof}
We proceed by induction on $n=ct(X)$. If $n$ is 1 or 2 then
$P_X(t)$ is 1 or 2, respectively,
%, resp.\ $P_X(t)=2$, 
and the inequality is trivial. 
%For $n=3,H_{1}(X)$ is at most of rank $1$ and $H_{0}(X)$ at most of
% rank $3$ and the formula is still valid. 

If $X=\bigcup_{k=1}^n X_k$, 
set $Y=\bigcup_{k\ne1}X_k$ and note that $ct(Y)$, $ct(Y\cap X_1)$ are at
most $n-1$. From the Mayer-Vietoris sequence for $X=X_1\cup Y$, 
\begin{equation*}
H_k(X_1)\oplus H_k(Y)\to H_k(X)\to H_{k-1}(X_1\cap Y),
\end{equation*}
and the inductive hypothesis, we see that $h_0(X)\le n$ and for $k>0$ 
\begin{equation*}
h_k(X) \le h_k(Y)+h_{k-1}(X_1\cap Y) \le \binom{n-2}{k}+\binom{n-2}{k-1} = 
\binom{n-1}{k}. \qedhere
\end{equation*}
\end{proof}

\smallskip
\begin{remark}
The lower bounds in Theorem \ref{Poincare} are not optimal
for non-connected spaces, such as discrete sets, because
the covering type of a non-connected space is the sum of
the covering types of its components. For this reason, we
shall concentrate on the covering type of connected spaces.
In particular, we will see that that the bound in Theorem \ref{Poincare}
is optimal when $X$ is 1-dimensional and connected.
\end{remark}

\section{Graphs}\label{sec:graphs}

Every 1--dimensional CW complex is a graph, and every connected graph is
homotopy equivalent to a bouquet of $h=1+E-V$ circles, where
$V$ and $E$ are the number of vertices and edges, respectively.
Since $n=ct(X)$ is an integer, the bound in Theorem \ref{Poincare}, 
that $\binom{n-1}2\ge h$, is equivalent to the inequality
$n\ge\lceil x\rceil$, where $x=\frac{3+\sqrt{1+8h}}2$ and
%$n\ge\lceil\frac{3+\sqrt{1+8h}}2\rceil$, where
the {\it ceiling} $\lceil{x}\rceil$ of $x$ %a real number $x$
denotes the smallest integer $\ge x$.

\begin{proposition}\label{ct:bouquets} 
When $X_h$ is a bouquet of $h$ circles then 
\[
ct(X_h)=\left\lceil{\frac{3+\sqrt{1+8h}}2\;}\right\rceil.
\]
That is, $ct(X_h)$ is the unique integer $n$ such that
\begin{equation*}
\binom{n-2}2 < h \le\binom{n-1}2.
\end{equation*}
\end{proposition}

\begin{proof}
By explicitly solving the displayed quadratic inequalities, we see
%The quadratic formula implies 
that the unique integer $n$ 
satisfying the displayed inequalities is
$\lceil\frac{3+\sqrt{1+8h}}2\rceil$. 
Theorem \ref{Poincare} implies that $ct(X_h)\ge n$. 
When $h=1$ (a circle),
we saw that the lower bound $ct=3$ is achieved. For $h=2,3$ the lower bound
is~4, and we saw in Example \ref{ct(figure8)} that it is also an upper
bound, so $ct(X_h)=4$ in these cases.

When $h$ is 6 or 10 we show that $ct(X_h)$ is 5 or 6, respectively, by
generalizing the pattern of Example \ref{ct(figure8)} to introduce more
L-shaped lines into the interior of a triangle, as shown in 
Figure \ref{fig:h=5,6}. For a bouquet $X_h$ of $h=\binom{n-1}2$ circles, 
the same construction (using $n-3$ internal L-shaped lines) 
shows that $ct(X_h)=n$.
If $h$ and $n$ satisfy the strict inequality of the proposition, 
we can construct a complex $X'$ like Figure \ref{fig:h=5,6} for 
$\binom{n-1}{2}$ circles and erase portions of the
interior lines to obtain a complex $X$ homotopic to a bouquet of 
$h$ circles, as illustrated by Figure \ref{fig:7-9 circles}. 
This shows that $ct(X_{h})\leq n$, as required.
\end{proof}
\begin{center}
\begin{figure}
%$ct(X_h)=5,6$ for bouquets of $6$ and 10 circles \newline
\setlength{\unitlength}{1cm} 
\begin{picture}(2,2.2)
\put(0,0){\line(0,1){2}}
\put(0,0){\line(1,0){2}}
\put(2,0){\line(-1,1){2.0}}
\put(0,1.34){\line(1,0){.66}}
\put(.66,0){\line(0,1){1.34}}
\thicklines
\put(0,.66){\line(1,0){1.34}}
\put(1.34,0){\line(0,1){.66}}
\end{picture}
\qquad 
\begin{picture}(2,2.2)
\put(0,0){\line(0,1){2}}
\put(0,0){\line(1,0){2}}
\put(2,0){\line(-1,1){2.0}}
\put(0,1){\line(1,0){1}}
\put(1,0){\line(0,1){1}}
\put(0,1.5){\line(1,0){.5}}
\put(0.5,0){\line(0,1){1.5}}
\thicklines
\put(0,.5){\line(1,0){1.5}}
\put(1.5,0){\line(0,1){.5}}
\end{picture}\par
\caption{$ct(X_h)=5,6$ for bouquets of $6$ and 10 circles}
\label{fig:h=5,6} 
\end{figure}
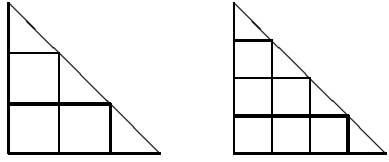
\end{center}
\begin{center}
\begin{figure}
%$ct(X_h)=6$ for bouquets of 7--9 circles
%\\ 
\setlength{\unitlength}{1cm}
\begin{picture}(2,2.2)
\put(0,0){\line(0,1){2}}
\put(0,0){\line(1,0){2}}
\put(2,0){\line(-1,1){2.0}}
\put(.5,1){\line(1,0){.5}}
\put(1,0){\line(0,1){1}}
%\put(.5,1.5){\line(1,0){.5}}
\put(0.5,0){\line(0,1){1.5}}
\thicklines
\put(.5,.5){\line(1,0){1.0}}
\put(1.5,0){\line(0,1){.5}}
\end{picture}
\qquad
\begin{picture}(2,2.2)
\put(0,0){\line(0,1){2}}
\put(0,0){\line(1,0){2}}
\put(2,0){\line(-1,1){2.0}}
\put(.5,1){\line(1,0){.5}}
\put(1,0){\line(0,1){1}}
%\put(0,1.5){\line(1,0){.5}}
\put(0.5,0){\line(0,1){1.5}}
\thicklines
\put(0,.5){\line(1,0){1.5}}
\put(1.5,0){\line(0,1){.5}}
\end{picture}
\qquad
\begin{picture}(2,2.2)
\put(0,0){\line(0,1){2}}
\put(0,0){\line(1,0){2}}
\put(2,0){\line(-1,1){2.0}}
\put(0,1){\line(1,0){1}}
\put(1,0){\line(0,1){1}}
%\put(0,1.5){\line(1,0){.5}}
\put(0.5,0){\line(0,1){1.5}}
\thicklines
\put(0,.5){\line(1,0){1.5}}
\put(1.5,0){\line(0,1){.50}}
\end{picture}\par
\caption{$ct(X_h)=6$ for bouquets of 7--9 circles}
\label{fig:7-9 circles}
\end{figure}
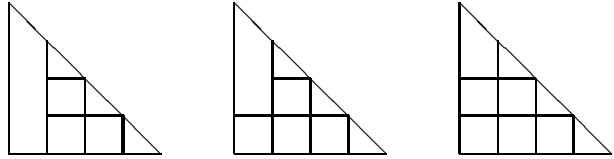
\end{center}

\goodbreak\medskip
\section{Classical surfaces}\label{sec:surfaces}

%For connected $CW$ complexes, the 
%question of determining the covering type arises naturally.
The literature contains some pretty upper bounds for the covering
type of an oriented surface of genus $g$ (a torus with $g$ holes).
We have already seen in Example \ref{sphere} that the 2-sphere has
covering type~4; this is the case $g=0$.

\smallskip
\paragraph{\it Torus}
For the torus $T$ (the case $g=1$), an upper bound $ct(T)\le7$
comes from the 7-country map in Figure \ref{fig:torus}, 
first described by Haewood in 1890 \cite{Hae} in connection with
the map coloring problem. The map is dual to an embedding of the 
complete graph on 7 vertices in the torus, so
each of the 7 countries is a hexagon, each pair meets in a face,
and any three countries meet in a point. 
We shall see in Theorem \ref{ct:torus} below that the covering
type of the torus is indeed~7.

%\smallskip%\goodbreak
\begin{center}
\begin{figure}
%Torus with 7 regions 
\includegraphics[height=4cm]{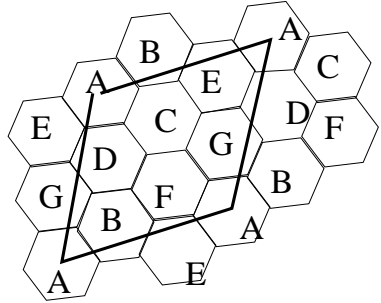} 
\par
\caption{Torus with 7 regions}\label{fig:torus}
\end{figure}
\end{center}

\begin{center}
\begin{figure}
%\\[0pt]\ \\
\includegraphics[height=4.3cm]{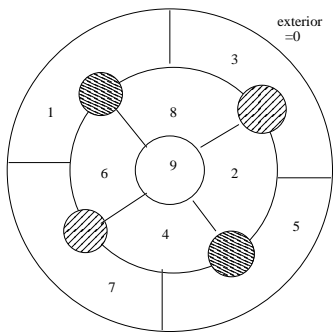} \par
\caption{Surface of genus 2, covered by 10 regions}
\label{fig:g=2}
\end{figure}
\end{center}

%\smallskip
%\begin{ex}\label{g=2}
\paragraph{\it Genus 2}
An upper bound for $S_2$ (the oriented surface of genus~2) is~10.
This comes from the example of a 10-region good covering of $S_2$
given by Jungerman and Ringel in \cite[p.\,125]{JR}, and reproduced in
Figure \ref{fig:g=2}.  To construct it, 
we start with a 10-region map on the 2-sphere (countries labelled 0--9),
cut out the opposing shaded circles and identify their edges.
We do not know the precise value of $ct(S_2)$.
We will see in Proposition \ref{ct>5} below that $ct(S_2)\ge6$,
so $6\le ct(S_2)\le10$.
%\end{ex}

%\section{Special surfaces}\label{sec:special}
\smallskip
Before giving more precise lower bounds for $ct(X)$, 
we record a simple result, which shows that $ct(X)>4$ for every 
closed surface $X$ except the 2-sphere.  We shall write $H^n(X)$ for 
the cohomology of $X$ with coefficients in a fixed field.

\begin{lemma}\label{lem:cup}
Suppose that $X=A\cup B$ and that $H^1(A)=H^1(B)=0$.
Then the cohomology cup product $H^1(X)\times H^1(X)\to H^2(X)$ is zero. 
\end{lemma}

\begin{proof}
Fix $x,y\in H^1(X)$. Since $x$ vanishes in $H^1(A)$, it lifts to
an element $x'$ in $H^1(X,A)$. Similarly, $y$ lifts to an element
$y'$ in $H^1(X,B)$. Then $x\cup y$ is the image of $(x',y')$
under the composition
\[
H^1(X,A) \times H^1(X,B) \map{\cup} H^1(X,A\cup B) \map{} H^1(X),
\]
where $\cup$ is the relative cup product; see \cite[p.\,209]{Hatcher}
Since $X=A\cup B$, $H^1(X,A\cup B)=0$, so $x\cup y=0$.
\end{proof}

\begin{proposition}\label{ct>5}
If the cohomology cup product is nonzero on $H^1(X)$,
then $ct(X)\ge6$.
\end{proposition}

\begin{proof}
The assumption implies that $H^2(X)\ne0$, so $ct(X)\ge4$
by Theorem \ref{Poincare}.
If the covering type were 4, let $A$ be the union of the first two
and $B$ the union of the last two subspaces. Since $A$ and $B$ are
homotopic to either one or two points, Lemma \ref{lem:cup} implies
that the cup product is zero in $H^*(X)$, contrary to fact.
This shows that $ct(X)\ne4$.

Now suppose that $\{ X_i\}_{i=1}^5$ is a good cover of $X$.
If $X_i\cap X_j\cap X_k=\varnothing$ for all $i,j,k$ then 
the \v{C}ech complex associated to the cover would have 
zero in degree~2. Since the cohomology of this complex 
is $H^{\ast}(X)$, that would contradict the assumption that
$H^2(X)\ne0$. By Example \ref{rem:ct=3}, there exists $i,j,k$
so that $A=X_i\cup X_j\cup X_k$ has $H^1(A)=0$. Let $B$ be
the union of the other two $X_m$; we also have $H^1(B)=0$.
By Lemma \ref{lem:cup}, the cohomology product is zero.
This contradiction shows that $ct(X)\ne5$.
\end{proof}

\begin{theorem}\label{ct:torus}
The covering type of the torus $T$ is exactly 7.
\end{theorem}

\begin{proof}
By Figure \ref{fig:torus} and Proposition \ref{ct>5},
the covering type of $T$ is either $6$ or $7$. 
%Theorem \ref{ct-surface} and Proposition \ref{ct>5}. 
Suppose the covering type were $6$, i.e., that $T$ had
a good cover by 6 subcomplexes $X_i$, $i=1,\dots,6$.
%$X_{1},X_{2},X_{3},X_{4},X_{5},X_{6}.$
As in the proof of Proposition \ref{ct>5}, the fact that
$H^2(T)\ne0$ implies that there exist $i,j,k$ 
%subcomplexes $X_{i},X_{j},X_{k}$ 
such that $X_{i}\cap X_{j}\cap X_{k}\ne\varnothing.$
By Example \ref{rem:ct=3},
$H^1(X_{i}\cup X_{j}\cup X_{k})=0$.
%
%Otherwise, by Example \ref{rem:ct=3}, this would imply that
%for any $i,j,k$ we have 
%$X_{i}\cap X_{j}\neq \varnothing$ and 
%
%Therefore, the \v{C}ech complex associated to the cover 
%(whose cohomology is $H^{\ast}(T))$ would have zero in degree~2
%be%
%\begin{equation*}
%0\rightarrow k^{6}\rightarrow k^{15}\rightarrow 0
%\end{equation*}%
%which is clearly impossible. 
Therefore, after reordering the indices of the
cover, we may assume that $H^{1}(B)=0,$ with $B=X_{4}\cup X_{5}\cup X_{6}.$
Set $A=X_{1}\cup X_{2}\cup X_{3}$, 
%we can finish the proof as for the other cases.
and note that, since $\dim H^1(A)\le1$ and $\dim H^1(X)=2$, 
there is an $x$ in the kernel of $H^1(X)\to H^1(A)$. 
Lifting it to an element $x'$ in $H^1(X,A)$, and lifting an
independent element $y$ of $H^1(X)$ to $H^1(X,B)$, the proof of Lemma
\ref{lem:cup} shows that $x\cup y=0$ in $H^2(X)$, which is not the case.
\end{proof}

\goodbreak
\smallskip
If $X$ is a non-orientable surface, its {\it genus} $q$ is the
dimension of $H^1(X,\Z/2)$. Up to homeomorphism, there is a unique
non-orientable surface $N_q$ of genus $q$ for each integer $q\ge1$,
with $N_1$ the projective plane and $N_2$ the Klein bottle.

\smallskip\goodbreak
%We now turn to non-oriented surfaces.

\paragraph{\it Projective plane}
Here is a pretty construction of %We now describe 
a good covering of the projective plane $\RP^2$ by 6 regions.
% since we saw in Proposition \ref{ct>5} that 
%the covering type is at least 6, we have $ct(\RP^2)=6$.
%
The antipode on the 2-sphere sends the regular dodecahedron to itself
(preserving the 20 vertices, 30 edges edges and 12 faces);
the quotient by this action defines a good polyhedral covering of
the projective plane $\RP^2$ by 6 faces (with 10 vertices and 15 edges).
%This shows that the covering type of $\RP^2$ is at most 6.

This construction yields 6-colorable map on $\RP^2$ whose 
regions are pentagons; this is illustrated in Figure \ref{fig:Petersen}.
We remark that the boundary of this polyhedral cover is the 
{\it Petersen graph} (see \cite[Ex.\,4.2.6(5)]{Gross.Tucker}).
\begin{center}
\begin{figure}
%A good covering of the projective plane
%\setlength{\unitlength}{1cm}
%\begin{picture}(4,4,4)\thicklines
%\put(0,0){\vector(0,1){2}}
%\put(0,2){\vector(0,1){2}}
%\put(4,0){\vector(-1,0){2}}
%\put(2,0){\vector(-1,0){2}}
%\put(4,4){\vector(0,-1){2}}
%\put(4,2){\vector(0,-1){2}}
%\put(0,4){\vector(1,0){2}}
%\put(2,4){\vector(1,0){2}}
%\put(0,0){\line(1,1){1.3}}
%\put(4,0){\line(-1,1){1.3}}
%\put(2,0){\line(0,1){0.6}}
%\put(1,4){\line(1,-2){.72}}
%\put(3,4){\line(-1,-2){.72}}
%\put(2,.6){\line(-1,1){.7}}
%\put(2,.6){\line(1,1){.7}}
%\put(1.3,1.3) {\line(1,3){.42}}
%\put(2.7,1.3){\line(-1,3){.42}}
%\put(1.72,2.55){\line(1,0){.56}}}
%\put(.5,2){1}
%\put(2,3.5){2}
%\put(3.5,2){3}
%\put(1.9,1.7){4}
%\put(1,.4){5}
%\put(3,.4){6}
%\end{picture}\par
\includegraphics[height=4.3cm]{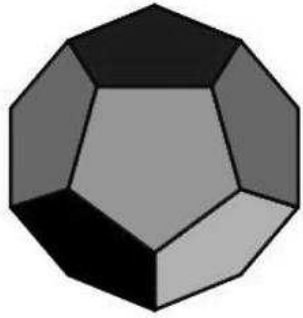} \par
\caption{The projective plane, covered by 6 regions}
\label{fig:Petersen}
\end{figure}
\end{center}

\begin{theorem}
The covering type of $\RP^2$ is $6.$
\end{theorem}

\begin{proof}
The good cover associated to the embedding of the Petersen graph
(giving the 6-coloring in Figure \ref{fig:Petersen}) 
shows that $6$ is an upper bound for the
covering type. Since the cup product on $H^1(\RP^2;\Z/2)$ is
nonzero, $ct(\RP^2)\ge6$ by Proposition \ref{ct>5}.
\end{proof}

%%%%%%%%%%%%%%%%%%%%%%%%%%%%%%%%%%%%%%%%%%%%%%%%%%%

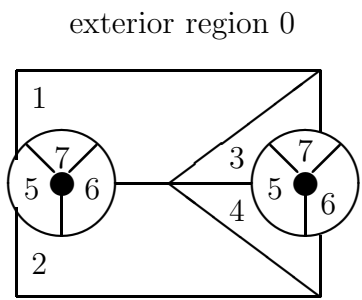
\begin{figure}
%A good covering of the Klein bottle
\setlength{\unitlength}{1cm}
\begin{picture}(4,4,4)\thicklines
\put(.7,3.5){exterior region 0}
\put(0,0){\line(1,0){4}}      % top & bottom
\put(0,3){\line(1,0){4}}
\put(0,0){\line(0,1){1.18}}   % left edge
\put(0,3){\line(0,-1){1.18}}
\put(4,0){\line(0,1){.81}}     % right edge
\put(4,3){\line(0,-1){.81}}
\put(2,1.5){\line(4,3){2}} %middle slants
\put(4,0){\line(-4,3){2}}
\put(3.8,0.8){\line(0,1){.6}}  % right 
\put(1.3,1.5){\line(1,0){.4}} %middle horizontals at y=1.5
\put(1.7,1.5){\line(1,0){.3}}
\put(2.0,1.5){\line(1,0){1.1}}  %added 2/17

\put(.6,1.5){\circle{1.4}}  %left circle
\put(.6,1.5){\circle*{.3}}
\put(.6,1.35){\line(0,-1){.55}} %diags
\put(3.8,1.5){\circle{1.4}}   %right circle
\put(3.8,1.5){\circle*{.3}}
\put(3.8,1.35){\line(0,-1){.55}}
\put(.5,1.65){\line(-1,1){.36}}
\put(3.7,1.65){\line(-1,1){.36}}
\put(0.7,1.65){\line(1,1){.36}}
\put(3.9,1.65){\line(1,1){.36}}
\put(.2,2.5){1}    %labels
\put(.2,0.3){2}
\put(2.8,1.7){3}
\put(2.8,1.0){4}
\put(.1,1.3){5}
\put(3.3,1.3){5}
\put(0.9,1.3){6}
\put(4.0,1.1){6}
\put(.5,1.7){7}
\put(3.7,1.8){7}
\end{picture}\par
\caption{A good covering of the Klein bottle, 8 regions}
\label{fig:Klein}
\end{figure}

\smallskip
\paragraph{\it Klein bottle}
A good covering of the Klein bottle by 8 regions is illustrated in 
Figure \ref{fig:Klein}.  To get it, we start with a cover of the
2-sphere by the five regions 0--4, add three regions in a circle
along the intersection of regions 0, 1 and 2 (labelled 5--7), and another 
three regions in a circle at the intersection of regions 0, 3 and 4 
(also labelled 5--7). Then cut out the two dark circles and 
identify their boundaries
as indicated. Note that region 7 does not meet regions 2 and 4,
and regions 3 and 6 do not meet either.

We will see in Corollary \ref{ct:Klein} that the covering type of
the Klein bottle is either 7 or 8.

%As observed in
%Section \ref{sec:cohomology},  non-orientable surfaces are
%classified by their genus $q\ge1$. 
% and cohomology with coefficients $\Z/2$.
% The {\it genus} of a non-orientable surface is $q=\dim H^1(X,\Z/2)$
% and its Euler characteristic is $\chi=2-q$.
%The projective plane $\RP^2$ has genus~1, and the Klein
%botle has genus~2.

\begin{theorem}\label{Nq>7}
%If $X$ is a non-oriented surface of genus
If $q\ge2$, then $ct(N_q)\ge7$.
\end{theorem}

\begin{proof}
When $q\ge2$, it is well known that the cup product is nontrivial on
$H^1(N_q,\Z/2)$.  For example, $H^1(N_2,\Z/2)$ is 2-dimensional,
and has a basis $\{ x,y\}$ such that $y^2=0$, $x^2=x\cup y\ne0$. 
By Proposition \ref{ct>5}, the covering type is at least 6.
%Thus the cup product to $H^2(N_2,\Z/2)$ is nontrivial; by

Suppose that $X=N_q$ had a good cover with 6 regions $X_i$.
As in the proof of Proposition \ref{ct>5}, the fact that
$H^2(X)\ne0$ implies that there exist $i,j,k$ 
such that $X_{i}\cap X_{j}\cap X_{k}\ne\varnothing.$
Set $B=X_{i}\cup X_{j}\cup X_{k}$, and note that
$H^1(B)=0$ by Example \ref{rem:ct=3}. Let $A$ be the union
of the other three subcomplexes in the cover.
As in the proof of Theorem \ref{ct:torus}, $\dim H^1(A)\le1$
so there is an element $u$ in the kernel of $H^1(X)\to H^1(A)$.
By inspection, there is an element $v\ne u$ in $H^1(X)$ so
that $u\cup v\ne0$. Lifting $u$ to $u'\in H^1(X,A)$ and $v$ to
$v'\in H^1(X,B)$, the proof of Lemma \ref{lem:cup} shows that
$u\cup v=0$, which is a contradiction.
\end{proof}

\begin{corollary}\label{ct:Klein}
The Klein bottle $N_2$ has covering type 7 or 8.
\end{corollary}

\begin{proof}
Combine Theorem \ref{Nq>7} with the upper bound  $ct(N_2)\!\le\!8$
coming from Figure \ref{fig:Klein},
which shows that.
\end{proof}

\section{Bounds via genus}\label{sec:cohomology}

We may regard the covering type and the chromatic number of a surface $S$
as functions of the genus of $S$. We will now show that the functions
are asymptotic; for both orientable and non-orientable surfaces,
their ratio lies between~1 and~$\sqrt{1/3}$ 
as the genus goes to $\infty$.
%the covering type goes to $\infty$ when the genus $g$ 
%goes to $\infty$ at the same rate as the chromatic number does. 

\begin{defn}\label{def:chromatic}
The {\it chromatic number} $\chr(S)$ of a surface $S$ is defined to be the
smallest number $n$ such that every map on $S$ is colorable with
$n$ colors. By a {\it map} on $S$ we mean a decomposition of $S$ into
closed polyhedral regions, called {\it countries}, such that 
%the interior of each region is open and 
the boundaries of the regions form a finite graph.
\end{defn}

\subsection*{\bf Orientable surfaces:}
When $S_g$ is an oriented surface of genus $g$,
it is a famous theorem that
\[
\chr(S_g) = \left\lfloor \frac{7+\sqrt{1+48g}}2\; \right\rfloor
\]
where $\lfloor x\rfloor$ denotes the greatest integer at most $x$.
The case $g=0$ is the Four-color Theorem; 
see \cite{Ringel} or \cite[Chap. 5]{Gross.Tucker} when $g>0$.

This integer is sometimes called the 
\textit{Heawood number} of $S_{g}$, after Percy Heawood who 
first studied $\chr(S_{g})$ in the 1890 paper \cite{Hae}.
We now show that $ct(S_g)\le1+\chr(S_g)$.

\begin{proposition}\label{ct-surface} 
The covering type $ct(S_g)$ of an oriented surface $S_g$ of
genus $g\ne2$ satisfies 
\begin{equation*}
\left\lceil\frac{3+\sqrt{1+16g}}{2}\right\rceil 
\le ct(S_g) \le 
\left\lceil \frac{7+\root{}\of{1+48g}}2\;\right\rceil.
\end{equation*}
\noindent If $g=2$, we have $6\le ct(S_2)\le10$.
\end{proposition}

\begin{proof}
Set $n=ct(X)$.
 For $g=0,1$ we have seen that $n=4,7$, respectively.
For $g=2$, we saw that $n\le10$ (Figure \ref{fig:g=2}),
and we saw in Proposition \ref{ct>5} that $n\ge6$.
Thus we may assume that $g>2$.

The lower bound $n\ge(3+\sqrt{1+16g})/2$ is just the solution 
of the quadratic inequality for $h_{1}=2g$ in Proposition \ref{Poincare}:
\begin{equation*}
4g\le 2{\textstyle\binom{n-1}2} = n^{2}-3n+2.
\end{equation*}

When $g>2$, Jungermann and Ringel showed in \cite[Thm.\,1.2]{JR}
that there is a triangulation of $S_g$ with %whose graph $\Gamma$ has
$n=\lceil\frac{7+\sqrt{1+48g}}2\;\rceil$ vertices,
$\delta=2n+4(g-1)$ triangles and
$\binom{n}2-t$ edges, where $t=\binom{n-3}2-6g$.
% there are $\delta=2n+4(g-1)$ triangles in all.  
The open stars of the $n$ vertices form a good open covering of $S_g$.
\end{proof}

\begin{example} ($g=3$) 
Proposition \ref{ct-surface} yields $6\le ct(S_3)\le10$.
A 10-vertex triangulation of $S_3$ is implicitly given by the 
orientation data on \cite[p.\,23]{Ringel}; the triangulation has
42 edges and 28 triangles.
\end{example}

\begin{remm}
Any triangulation of $S_g$ determines a graph $\Gamma$ on $S_g$.
The dual graph of $\Gamma$ is formed by taking a vertex in the
center of each of the $\delta$ triangles, and 
connecting vertices of adjacent triangles along an arc 
through the edge where the triangles meet. Each country $X_v$ in the 
dual map is a polygonal region, containing exactly one vertex $v$ from 
the original triangulation, and the number of sides in 
the polygon $X_v$ is the valence of $v$ in the $\Gamma$.

The $n$ countries in the dual map form a good closed covering of $S_g$
because the intersection of $X_v$ and $X_w$ is either a face (when
$v$ and $w$ are connected) or the empty set, and three polygons
meet in a vertex exactly when the corresponding vertices
form a triangle in the original triangulation.
\end{remm}

\subsection*{\bf Non-orientable surfaces:}
Similar results hold for non-orientable surfaces.
For example, consider the two non-orientable surfaces of genus $q\le2$. 
The projective plane has chromatic number 6; see Figure \ref{fig:Petersen}.
It is also well known that the Klein bottle has chromatic number 6;
this fact was discovered by Frankin in 1934 \cite{Franklin}.
However, the 6-country map on the Klein bottle is not a good covering,
because some regions intersect in two disjoint edges;
see \cite[Fig.\,1.9]{Ringel}. Theorem \ref{Nq>7} shows that 
there are no good covers of the Klein bottle by 6 regions.
In this case, $ct(X)$ is strictly bigger that the chromatic number of $X$.

For $q\ne1,2$, a famous theorem (see \cite[Thm.\,4.10]{Ringel}) states that 
the chromatic number of a non-orientable surface of genus $q$ is 
\[
\chr(N_q)=\lfloor \frac{7+\sqrt{1+24q}}2\rfloor.
\]
We now show that the covering type of a non-orientable surface $N_q$
grows (as a function of $q$) 
at the same rate as the chromatic number of $N_q$.
In particular, $ct(N_q)\le1+\chr(N_q)$.

\begin{proposition}\label{ct:non-oriented} 
The covering type $ct(N_q)$ of a non-oriented surface $N_q$ of
genus $q\ge4$ satisfies 
\begin{equation*}
5\le \left\lceil\frac{3+\sqrt{1+8q}}{2}\;\right\rceil 
\le ct(N_q) \le 
\left\lceil \frac{7+\sqrt{1+24q}}2\;\right\rceil.
\end{equation*}
\end{proposition}

\begin{proof}
The lower bound is immediate from the solution of the quadratic
%the quadratic formula applied to the 
inequality $\binom{n-1}2\ge q$ %$\binom{n-1}3\ge1$ 
of Theorem \ref{Poincare}. 

The upper bounds were investigated by Ringel in 1955 \cite{Ringel1955}; 
he considered polyhedral covers of $N_q$ by $\lambda$
countries, any two meeting in at most an arc and any three
meeting in at most a point, and showed that
$\lambda \ge\lceil\frac{7+\sqrt{1+24q}}2\rceil$,
with the lower bound being achieved for all $q\ne2,3$.
%For $q=2$ (resp., $3$) he produced covers with $8$ (resp., $9$)
%polyhedra.
\end{proof}

\begin{rem}
The upper bounds for $q<4$ found by Ringel in \cite[p.\,320]{Ringel1955} were:
$ct(N_2)\le8$ if $q=2$ (the Klein bottle), and $ct(N_3)\le9$ if $q=3$.
We saw in Corollary \ref{ct:Klein} that $ct(N_2)$ is 7 or 8; 
Ringel's upper bound for the Klein bottle 
corresponds to Figure \ref{fig:Klein}.

%\begin{remark}\label{ct:q=3}
%If $N_q$ is any non-oriented surface of genus $q>2$, the proof of 
%Theorem \ref{ct:Klein}
%applies to show that $ct(N_q)\ge7$, because the cup product 
%is nontrivial on $H^1(N_q,\Z/2)$.
Combining Theorem \ref{Nq>7} with Proposition \ref{ct:non-oriented} and
Ringel's bound for $q=3$, we see that
$ct(N_3)$ and $ct(N_4)$ are either 7, 8 or 9 and that
$7\le ct(N_5)\le10$.
\end{rem}

%Write $N_q$ for a non-orientable surface of genus $q$. Since $\dim
%H^2(N_q)=1$, and the cup product is nontrivial on $H^1$, we know that
%the covering type of $N_q$ is at least 5, by Proposition \ref{ct>5}.

\section{Higher dimensions}\label{sec:higher}

We do not know much about the covering type of higher-dimensional
spaces. In this section, we give a few general theorems for upper bounds
on the covering type.

\subsection*{\bf Suspensions:}
The covering type of the suspension of a finite CW complex $X$ is 
at most one more than the covering type of $X$.

To see this, recall that the cone $CX$ is the quotient of 
$X\times[0,1]$ by the relation $(x,1)\sim(x',1)$ for all $x,x'\in X$, 
while the suspension $SX$ is the quotient of $X\times[-1,1]$ by the 
two relations $(x,1)\sim(x',1)$ and $(x,-1)\sim(x',-1)$.
If we start with a good closed cover of $X$ by subcomplexes $X_i$, 
the cones $CX_i$ of these form a good cover of the cone $CX$. 
Viewing the suspension $SX$ as the union of an upper cone and a 
lower cone, the lower cones $CX_i$ together
with the upper cone form a good (closed) cover of $SX$.
(This argument works for a good open cover of any topological space,
provided we use open cones of the form
$X\times(-\varepsilon,1]/(x,1)\sim(x',1)$;
we leave the details as an easy exercise.)

The operation of coning off a subspace generalizes the suspension.
To give a bound for the covering type in this case, 
we need a compatible pair of good covers.

\begin{theorem}\label{thm:cone}
Let $X$ be a subcomplex of a finite CW complex $Y$\!. 
Suppose that $Y$ has a good closed cover $Y_1,...,Y_n$ such that 
$\{X\cap Y_i\}$ is a good closed cover of $X$.
%Suppose that $X$ and $Y$ have good covers $X_1,...,X_n$ and 
%$Y_1,...,Y_p$ such that 
%(i) each $X_i$ is contained in some $Y_j$ and (ii) 
%the $\{X_i\cap Y_j\}$ also forms a good cover of $X$.
Then the covering type of the cone $Y\cup_XCX$ of the inclusion
$X\subset Y$ satisfies:
\begin{equation*}
ct(Y\cup_XCX)\le n+1.
\end{equation*}
In particular, the suspension $SX$ has $ct(SX)\le ct(X)+1$.
\end{theorem}

\begin{proof}
The cone $CX$ together with $\{Y_i\}_{i=1}^n$ 
forms a good closed cover of $Y\cup_X CX$, whence the first assertion. 
Since the suspension $SX$ is the cone of the inclusion of $X$
into the upper cone $Y$ of the suspension, the second assertion
follows from the observation that, given a good cover $\{X_i\}$
of $X$, the cones $Y_i$ of the $X_i$ satisfy the hypothesis of
the theorem.  %Theorem \ref{thm:cone}.
\end{proof}

\subsection*{\bf Covering spaces:}
Another simple comparison involves the covering type of a covering space. 
If $X$ is an $n$-sheeted covering space of $Y$, 
the covering type of $X$ is at most $n\cdot ct(Y)$, 
because the inverse image of a contractible
subcomplex of $Y$ is the disjoint union of $n$ contractible
subcomplexes of $X$.

Recall that the mapping cylinder, $\cyl(f)$, of a function $f:X\to Y$
is the quotient of $X\times[0,1]\cup Y$ by the equivalence relation
$(x,0)\sim f(x)$; $\cyl(f)\to Y$ is a homotopy equivalence.
We may define the cone of $f$, $\cone f$, to be
the cone of the inclusion of $A=X\times\{1\}$ into $\cyl(f)$.

\begin{theorem}\label{thm:covering}
If $f:X\to Y$ is an $n$-sheeted covering space, the
cone $\cone f$ has covering type at most $n\cdot ct(Y)+1$,
\end{theorem}

\begin{proof}
If $\{Y_k\}$ is a good cover of $Y$, then the inverse image of each $Y_k$
is a disjoint union of $n$ contractible subspaces we shall call $X_{ik}$,
and each $f_{ik}:X_{ik}\to Y_k$ is a homeomorphism.
Then the mapping cylinders $\cyl(f_{ik})$ form a good cover of $\cyl(f)$.
Since the $X_{ik}\times\{1\}=A\cap\cyl(f_{ik})$ form a good cover of 
$A=X\times\{1\}$, the conditions of Theorem \ref{thm:cone} are met.
Since $\cone f$ is homotopy equivalent to $CA\cup_A\cyl(f)$, the result follows.
\end{proof}

\begin{example}
The bound in Theorem \ref{thm:covering} is rarely sharp, although
the cone of the projection from $S^{0}$ to a point is $S^{1}$ and indeed
$ct(S^{1})=3$. On the other hand, if $X$ is $n$ discrete points and $f$
is the projection from $X$ to a point, the cone of $f$ is a bouquet of $n-1$ 
circles, whose covering type is given by Theorem \ref{ct:bouquets};
$ct(\cone f)$ is smaller than $n+1$ when $n\ge4$.

Another example is given by $\RP^m$, which is the cone of the 
degree~2 function $f:S^{m-1}\to S^{m-1}$; Theorem \ref{thm:covering} yields
the upper bound $ct(\RP^m)\le2m+3$. This is not sharp for $m=2$
since we know that $ct(\RP^2)=6$, but (with Proposition \ref{bound:m+2})
it does give the linear growth rate
\[
m+2\le ct(\RP^m)\le2m+3.
\]
\end{example}

\begin{rem}
We do not know a good upper bound for the covering type of a product,
beyond the obvious bound $ct(X\times Y)\le ct(X)\,ct(Y)$.
The torus $T=S^1\times S^1$ shows that this bound is not sharp,
since $ct(T)=7$ and $ct(S^1)=3$.
\end{rem}

\goodbreak
\smallskip
\paragraph{\textit{Acknowledgements}}
The authors would like to thank Mike Saks and Patrick Devlin 
for discussions about color mapping theorems, Ieke Moerdijk
for discussions about Quillen's Theorem~A, and the referee for
numerous expositional suggestions.

\end{document}